\documentclass{amsart}

\usepackage[colorlinks,
linkcolor=blue,
anchorcolor=blue,
citecolor=blue]{hyperref}

\newtheorem{theorem}{Theorem}[section]
\newtheorem{lemma}[theorem]{Lemma}
\newtheorem{corollary}[theorem]{Corollary}
\newtheorem{proposition}[theorem]{Proposition}
\theoremstyle{definition}
\newtheorem{definition}[theorem]{Definition}

\newtheorem{problem}[theorem]{Problem}
\theoremstyle{remark}
\newtheorem{remark}[theorem]{Remark}

\DeclareMathOperator{\wind}{wind}
\DeclareMathOperator{\ind}{ind}

\numberwithin{equation}{section}

\newcommand{\D}{\mathbb D}
\newcommand{\T}{\mathbb T}
\newcommand{\C}{\mathbb C}
\newcommand{\R}{\mathbb R}
\newcommand{\lb}{\lambda}
\begin{document}

\title[Spectral  properties of Bergman Toeplitz operators]{Spectral  properties of Toeplitz operators with harmonic function symbols  on the Bergman space}

\author[P. Cui]{Puyu  Cui}
\address{School of Mathematics, Liaoning Normal University, Dalian 116029,   China}
\email{cuipuyu1234@163.com}
\author[Y. Lu]{YuFeng LU}     %%%  2nd Author's info, if exists, or you may delete this part directly  %%%
\address{School of Mathematical Sciences, Dalian University of Technology, Dalian 116024, China}
\email{lyfdlut@dlut.edu.cn }
\author[R. Yang]{Rongwei Yang}     %%%  3nd Author's info, if exists, or you may delete this part directly  %%%
\address{Deparetment of Mathematics and Statistics, University at Albany, Albany, NY 12222, USA}
\email{ryang@albany.edu}
\author[C. Zu]{Chao Zu$^*$ }{\thanks{* Corresponding author}}      %%%  4nd Author's info, if exists, or you may delete this part directly  %%%
\address{School of Mathematical Sciences, Dalian University of Technology, Dalian 116024, China}
\email{zuchao@dlut.edu.cn}

\thanks{The first author was supported  by the Scientific Research Fund of Liaoning Provincial Education Department of China  (Grant No. LJKMZ20221405), the  fourth  author was supported   by the National Natural Science Foundation of China (Grant No.12031002), 
the  fourth  author was supported  by the National Natural Science Foundation of China (Grant No.12401151), and the Postdoctral Researcher Foundation of China (Grant No.GZB20240100).}

%\author{Author Two}
%\address{Mathematical Research Section, School of Mathematical Sciences,
%Australian National University, Canberra ACT 2601, Australia}

%    General info
\subjclass[2010]{ 47B35, 47B38}

%\date{January 1, 2001 and, in revised form, June 22, 2001.}


\keywords{Bergman space,  Toeplitz operator, harmonic polynomial symbol, spectrum}

\begin{abstract}
This paper investigates the spectral properties of Toeplitz operators on the Bergman space of unit disk.
We present an integral representation of $ T^*_{z^m}$, which establishes a connection between the Bergman functions and the solutions of PDE theory.
In fact, by leveraging the Poincaré  theorem in difference equations and the solution forms of differential equations, this paper describes the kernels  of certain Toeplitz operators with harmonic polynomial symbols, and further gives the sufficient conditions for the connectedness of  the spectra of these  Toeplitz operators.
The spectral properties of $ T_\varphi$~with~$\varphi (z) =\overline{z}^{m} + \alpha z^m + \beta$ are characterized,  such as $\sigma(T_\varphi)= \overline{\varphi (\mathbb {D})}$, Fredholm index of $T_\varphi$ can only be one of  $m,~-m$~and~$0$, $T_\varphi$ satisfies Coburn's theorem.
These findings offer an illuminating example for the essential projective spectra of non-commuting operators.
\end{abstract}

\maketitle

%\section*{This is an unnumbered first-level section head}
%This is an example of an unnumbered first-level heading.

%% The correct journal style for \specialsection is all uppercase; a known bug
%% in amsart.cls prevents this, so input must be uppercase until it is fixed.
%\specialsection*{This is a Special Section Head}
%\specialsection*{THIS IS A SPECIAL SECTION HEAD}
%This is an example of a special section head%
%%%%%%%%%%%%%%%%%%%%%%%%%%%%%%%%%%%%%%%%%%%%%%%%%%%%%%%%%%%%%%%%%%%%%%%%
%\footnote{Here is an example of a footnote. Notice that this footnote
%text is running on so that it can stand as an example of how a footnote
%with separate paragraphs should be written.
%\par
%And here is the beginning of the second paragraph.}%
%%%%%%%%%%%%%%%%%%%%%%%%%%%%%%%%%%%%%%%%%%%%%%%%%%%%%%%%%%%%%%%%%%%%%%%%

\section{Introduction}
The classical Bergman space $L_a^2(\mathbb{D})$ is the closed subspace consisting of all analytic functions in the Hilbert space $L^2(\mathbb{D}, dA),$ where $dA$ refers to the normalized Lebesgue area measure on the unit disk $\mathbb{D}$. For $f, g\in  L^2(\mathbb{D},dA),$ the inner product is
\[ \langle f,g \rangle=\int_{\mathbb{D}}f(z)\overline{g(z)}dA.\]
For $\varphi\in L^\infty(\mathbb{D},dA),$ the collection of all essentially bounded measurable functions, the Toeplitz operator $T_\varphi$ on $L_a^2(\mathbb{D})$ with symbol $\varphi$ is defined by
\[T_\varphi(f)=P(\varphi f),\ \ \ f\in L_a^2(\mathbb{D}),\]
where $P$ is the orthogonal projection from $L^2(\mathbb{D},dA)$ onto $L_a^2 (\mathbb{D}).$ It is not hard to verify that the adjoint $T^*_\varphi=T_{\overline{\varphi}}$.

The classical Hardy space $H^2({\mathbb T})$ over the unit circle $\mathbb{T}$ is the closed subspace of $L^2({\mathbb T}, \frac{d\theta}{2\pi})$ consisting of functions that have analytic continuation into $\D$, and the Toeplitz operator on $H^2({\mathbb T})$ is defined parallelly.
The spectral theory of Toeplitz operators on Hardy space $H^2({\mathbb T})$ is relatively complete.
The spectra and essential spectra of bounded Toeplitz operators on $H^2({\mathbb T})$ are known to be path-connected \cite{Widom1, Widom2, Doug}.
In 1966, Coburn \cite{Coburn} demonstrated that every nonzero Toeplitz operator on $H^2(\T)$ has the property that either $T_\varphi$ is injective or $T^*_\varphi$ is injective. In this paper, we say that a Hilbert space linear non-scalar operator $T$ is of {\em Coburn type} if either $\ker T$ or $\ker T^*$ is equal to $\{0\}$.  Apparently, the Coburn type property is linked with the spectral property of Toeplitz operators. In the Hardy space setting, a great amount of information has been obtained, for instance see \cite{Doug}.
The case for the Bergman space is marked more challenging. The difficulty, to a large extent, is due to the fact that Bergman space functions usually don't have boundary values at $\T$. %Are there Toeplitz operators not of Coburn type?
Nevertheless, determining whether Toeplitz operators on $L_a^2 (\mathbb{D})$ are of Coburn type remains an appealing problem,
and some results have been obtained for Toeplitz operators with harmonic symbols. For example, it is shown only very recently that $T_{\overline{z} + c z^n}, n \geq 0, c \in \mathbb{C}$, are Coburn type \cite{Coburn Lee2023}.

The study of harmonic symbols constitutes a central aspect of the spectral theory of Toeplitz operators on $L_a^2 (\mathbb{D})$.
McDonald and Sundberg \cite{MC1979} showed that, in the case that symbols are either real harmonic or harmonic and piecewise continuous on $\T$, the essential spectra of Toeplitz operators are connected. Indeed, it was conjectured that the spectra of the Toeplitz operators with harmonic symbols are always path-connected \cite{conj}. However, some counterexamples were discovered later \cite{example, Invetible2020}. One is thus tempted to find subclasses of harmonic symbols for which the conjecture remains valid. As an effort along this line, Zhao and Zheng \cite{spec2016} showed that $\sigma(T_{\varphi})=\overline{\varphi(\D)}$ when $\varphi$ is of the form $\overline{z}+\alpha z+\beta, \alpha, \beta\in \C$. For more information on this topic, we refer the reader to \cite{DZ2020, Fa, Ka, Lue}.

This paper is organized as follows. In Section 2, we introduce a new integral formula for the operator $T_{\bar{z}}^m=(T^*_z)^m$ and recall some theorems used in this paper,
such as  Poincar\'{e}'s theorem, Perron's theorem and so on.
Section 3 describes the kernel and range of the Toeplitz operator $ T_{\overline{z}^m +f }$, where $f$ is  bounded analytic. The result implies that the Toeplitz operators  $T_{\overline{z}^m + cz^n}$, where $m \geq1,  n \geq 0, c \in \mathbb {C}$, are of Coburn type.
Using Poincar\'{e}'s theorem and Perron's theorem, the injective Toeplitz operators $T_{\overline{q}+p}$ with  analytic polynomials $p$ and $q$ are characterized in section 4.
The kernel of $T_{\overline{z}^{m} + \alpha z^m + \beta}$ with $m \ge 1, \alpha, \beta \in \mathbb {C}$ are also  illustrated.
The  connectedness of  spectrum $\sigma(T_\varphi)$ for the case $\varphi=\overline{q}+p$ are determined.
The last section  extends Zhao and Zheng's result to Toeplitz operators with symbols $\varphi={\overline{z}^{m} + \alpha z^m + \beta}$.
We  also  conclude that the Fredholm index of $T_{\varphi}$ is one of $m$,$-m$ and $0$.
Furthermore, based on the above facts, the definition of the essential projective spectra of non-commuting operators is proposed.

%-------------------------------------------------

\section{Preliminaries}

To simplify calculations in this paper, we start the preparation by mentioning a well-known fact.

\subsection{Spectral Properties}

Recall that $P: L^2(\mathbb{D}, dA)\to L^2_a(\mathbb{D})$ is the orthogonal projection.
\begin{lemma}\label{p}
For $m, k \geq 0$, we have
\begin{eqnarray*}
P(\overline{z}^mz^k)=   \left\{
                          \begin{aligned}
                           & \frac{k-m+1}{k+1}z^{k-m} \ \ & & k\geq m, \\
                           & 0 \ \ & & k < m.
                          \end{aligned}
                           \right.
\end{eqnarray*}
\end{lemma}
The following lemma is instrumental for characterizing the kernel and range of Toeplitz operator $T_{\overline{z}^m +f}$, where $f$ is bounded and analytic. The case $m=1$ is considered in \cite{pro19, pro20}.
\begin{lemma}\label{tool1}
Suppose that $g\in L^2_a({\mathbb D})$ and $g(0)= g'(0)=\cdots= g^{(m-1)}(0)=0$ for $m \geq 1$. Then
\begin{align}
 T^*_{z^m} g (z)&=\frac{1}{z^{m+1}}\int^z_0 \big[ w g'(w) -(m-1) g(w)\big] dw \\
               &= \frac{1}{z^m}g(z)-\frac{m}{z^{m+1}} \int^z_0 g(w)dw.
\end{align}
\end{lemma}
\begin{proof}%-----proof
This lemma can be verified by considering the power series expansion of $g$ and then using Lemma \ref{p}.
\end{proof}

For a general bounded linear operator, the following theorem in \cite{Pe} describes the connection between the spectrum and the essential spectrum.
\begin{theorem}\label{picture}
Let $T$ be a bounded linear operator on a Hilbert space $\mathcal{H}$ and $H$ be ``a
hole in $\sigma_e (T)$'' (which is a bounded component of $ \mathbb {C} \backslash \sigma_e (T)$ such that
 \[ \mathrm{ind} (T - \lambda I) = 0, \lambda \in H,\]
then $\lambda$ satisfies one of the three,

(i) $ H \cap \sigma(T) = \emptyset$,

(ii)  $ H \subset \sigma(T)$,

(iii) $ H \cap \sigma(T)$ is a countable set of isolated eigenvalues of $T$, each having finite multiplicity.

\noindent Moreover, the intersection of $ \sigma(T)$ with the unbounded component of $ \mathbb {C} \backslash \sigma_e (T)$ is a countable
set of isolated eigenvalues of $T$, each of which has finite multiplicity.
\end{theorem}

The following lemma from \cite{MC1979,  ess spe, zhu} characterizes the essential spectra of Toeplitz operators on the Bergman space.
\begin{lemma}\label{spec}
Assume $\varphi (z) \in C(\overline{ \mathbb {D}})$. Then $\sigma_e (T_\varphi) = \varphi (  \mathbb {T})$.
\end{lemma}
\noindent Furthermore, the Fredholm index of $T_\varphi$ is
\begin{eqnarray}
\mathrm{ind}(T_\varphi) = - \text{wind}( \varphi(  \mathbb {T}), 0 )= -\frac{1}{2\pi \text{i} }\int_{\varphi ( \mathbb T)} \frac{dw}{w},
\end{eqnarray}
where wind$( \varphi( \mathbb {T}, 0) )$ means the winding number of the closed oriented curve  $\varphi(\mathbb T$ surrounding $0$.

Combining Theorem \ref{picture} and Lemma \ref{spec}, one arrives at the following description of the spectra of Toeplitz operators.
\begin{theorem}\label{specT}
Let $\varphi$ be continuous on the closed unit disc $\overline{\D}$. Then
\begin{align*}
 \sigma (T_{\varphi}) &= \varphi (\mathbb T ) \bigcup \big\{ \lambda \in \mathbb {C}: \lambda \notin   \sigma_e(T_{\varphi}), \ind (T_{\varphi} - \lambda I ) \neq 0 \big  \}\\
 & \bigcup \big\{ \lambda \in \mathbb {C}:   \lambda \in  \sigma_p(T_{\varphi}),  \lambda \notin  \sigma_e(T_{\varphi}), \ind(T_{\varphi}- \lambda I)=0 \big\}.
\end{align*}
\end{theorem}

\subsection{Difference Equations}

A difference equation (also called a recurrence equation) express a term in the sequence (e.g. Fibonacci sequence) as a function of preceding terms. They are discrete analogs of differential equations and are essential for modeling and analyzing discrete-time processes. In this paper, a number of spectral properties of the Toeplitz operators are discovered by leveraging the Poincar\'{e}'s theorem in linear difference equations. For an introduction to this subject, we refer the reader to \cite{KP}. Making use of above theorems in the study of Toeplitz operators is a technical innovation of this paper.
\begin{definition}
A one-parameter homogeneous linear equation
\begin{equation}\label{1.0}
  u(t+d)+p_{d-1}(t)u(t+d-1)+\cdots+p_0(t)u(t)=0,\ \ \ t\geq 0,
\end{equation}
is said to be of \emph{Poincar\'{e}} if
\[ \lim_{t\rightarrow \infty} p_k(t)=c_k, \ \ \ 0\leq k\leq d-1,\] i.e., if every coefficient function $p_k$ converge to a finite constant as $t$ goes to infinity.
\end{definition}
The following theorems are Theorems 5.1, 5.2 and 5.3 in \cite{KP}.
\begin{theorem}[Poincar\'{e}'s Theorem]\label{thm1.1}
Assume Equation \eqref{1.0} is of Poincar\'{e}'s type, and the roots $\lambda_1,\cdots,\lambda_d$ of $\lambda^d+c_{d-1}\lambda^{d-1}+\cdots +c_0=0$ have distinct moduli. Then every nontrivial solution $u$ of Equation \eqref{1.0} satisfies
\[ \lim_{t\rightarrow \infty} \frac{u(t+1)}{u(t)}=\lambda_i\ \text{for some}\ i.\]
\end{theorem}

To  understand the solution structure of Equation \eqref{1.0}, the following two theorems are introduced.

\begin{theorem}[Perron's Theorem]
In Theorem \ref{thm1.1}, if the function $p_0$ is non-vanishing, then Equation \eqref{1.0} has $d$ independent solutions $u_1,\cdots, u_d$ that satisfy
 \[ \lim_{t\rightarrow \infty} \frac{u_i(t+1)}{u_i(t)}=\lambda_i,\ \ \ i=1,\cdots, d. \]
\end{theorem}

\begin{theorem}[Perron's Theorem]\label{thm1.3}
Suppose $\lim_{t\rightarrow \infty} \frac{u(t+1)}{u(t)}=\lambda$,
\begin{enumerate}
  \item if $\lambda \neq 0$, then $u(t)=\pm \lambda^t e^{z(t)}$  with $z(t)\ll t$,~~~($t\rightarrow \infty$);
  \item if $\lambda = 0$, then $|u(t)|= e^{-z(t)}$  with $z(t)\gg t$,~~~($t\rightarrow \infty$).
\end{enumerate}
\end{theorem}
%---------------------- zero number
\subsection{ Distribution of zeros for polynomials }

Fredholm index is an important spectral property of operators.
In this paper, the Fredholm  indices of Toeplitz operators on disk Bergman space with  some polynomial symbols are completely characterized by the distribution of zeros of polynomials.
\begin{definition}
Let $\{a_0 ,a_1 , \cdots, a_n \}$ be a sequence of real numbers. We define $N(a_0 ,a_1 ,\cdots, a_n )$
as the total number of variations of sign in the reduced sequence obtained by ignoring the element $0$.
% For example, $N(1,?1,2,?3) = 3$ and $N(2,?1,?3,0,2,0) = 2$ by the definition.
\end{definition}
The following two theorems refer to Corollary 11.5.14 and  Corollary 11.5.15 in \cite{zero dis}.
\begin{theorem}\label{dis}
Suppose that $n \geq 1$. Let $p(z)= a_n z^n + a_{n-1}z^{n-1} + \cdots + a_0$ be a polynomial of degree $n$.
Define the determinants for $k = 1,2, \ldots, n$ as
\begin{eqnarray}\label{zero inside}
 M_k := \mathrm{det}(B^*_kB_k - A^*_kA_k),
\end{eqnarray}
where
\[A_k =\begin{bmatrix}
  a_0 & a_1 & \cdots & a_{k-1}\\
   &  a_0  & \cdots & a_{k-2}\\
   & &\ddots & \vdots \\
    0 &  & & a_0
\end{bmatrix},
\]
\[B_k =\begin{bmatrix}
  \overline{a}_n & \overline{a}_{n-1} & \cdots & \overline{a}_{n-k+1}\\
   &  \overline{a}_n  & \cdots & \overline{a}_{n-k}\\
   & &\ddots & \vdots \\
    0 &  & & \overline{a}_n
\end{bmatrix}
\]
are two upper triangular matrices, and $A_k^*$ is the conjugate transpose of $A_k$.
If $M_k \neq 0$ for all $1 \le k \le n$, then the number of zeros inside the unit disk ${\mathbb D}$ of $p$ is given by $n - N(1, M_1 , \ldots, M_n).$
\end{theorem}

\begin{theorem}\label{dis2}
Let $p(z)= a_n z^n + a_{n-1}z^{n-1} + \cdots + a_0$  be a polynomial of degree $n~(n \geq 1)$.
Then $p(z)$ has all its zeros inside $D$ if and only if  $M_1, \cdots, M_n$  defined  in \eqref{zero inside} are all positive.
\end{theorem}
For the case $\varphi(z)= \alpha z^2+ \beta z +\gamma$.
\[ A_1 = \gamma, B_1=\overline{ \alpha},~ M_1 = \text{det}(B^*_1B_1 - A^*_1A_1)= |\alpha|^2 -|\gamma|^2. \]
\[A_2 =\begin{bmatrix}
  \gamma & \beta \\
 0  &    \gamma\\
\end{bmatrix},
\]
\[B_2 =\begin{bmatrix}
  \overline{\alpha} & \overline{\beta} \\
 0  &   \overline{ \alpha}\\
\end{bmatrix}.
\]
\[ B^*_2B_2 - A^*_2A_2
  = \begin{bmatrix}
          |\alpha|^2 & \alpha \overline{\beta} \\
          \overline{\alpha} \beta  &  |\beta|^2+|\alpha|^2  \\
              \end{bmatrix}
-    \begin{bmatrix}
       |\gamma|^2 & \beta\overline{\gamma} \\
      \overline{\beta}\gamma  &  |\beta|^2+ |\gamma|^2  \\
   \end{bmatrix}
      = \begin{bmatrix}
         |\alpha|^2 -|\gamma|^2 &\alpha \overline{\beta} -\beta\overline{\gamma} \\
       \overline{ \alpha }\beta - \overline{\beta}\gamma  &   | \alpha|^2-|\gamma|^2\\
       \end{bmatrix}. \]
\[  M_2 = \text{det}(B^*_2B_2 - A^*_2A_2) = (|\alpha|^2 -|\gamma| )^2- |\alpha \overline{\beta} - \beta \overline{\gamma}|^2.\]
Then the number of zeros inside the unit disk ${\mathbb D}$ of $\varphi$ is given by
\begin{eqnarray}\label{num}
2 - N\big(1,  |\alpha|^2 -|\gamma|^2, (|\alpha|^2 -|\gamma|^2 )^2- |\alpha \overline{\beta} - \beta \overline{\gamma}|^2 \big)
\end{eqnarray}
under the condition that $ |\alpha| \neq |\gamma| $ and $   (|\alpha|^2 -|\gamma|^2 )^2 \neq |\alpha \overline{\beta} - \beta\overline{\gamma}|^2$.

\section{Coburn type theorem for  $T_{\overline{z}^m + cz^n}$ }
%------------kernel-------
We let $H^\infty({\mathbb D})$ denote the space of bounded analytic functions on $\mathbb D$.
This section studies the kernel and range of Toeplitz operators $T_{\overline{z}^m + f}$, where $ f \in H^\infty({\mathbb D})$ and $m \ge 1$. The case $m=1$ is considered in \cite{Coburn Lee2023}. In the sequel, we assume $f(z) = \sum_{k=0}^{\infty}a_kz^k\in H^{\infty}({\mathbb D})$, and $g(z) =\sum^{\infty}_{k=0}d_kz^k \in L^2_a({\mathbb D})$. In order to facilitate the computation, we also write $g$ as a decomposition $g=g_0+g_1=g_0+z^mg_2$, where $g_0$ is the partial sum $\sum^{m-1}_{k=0}d_kz^k$.

\begin{lemma}\label{thm ker}
The following equations are equivalent:
%----------u_2(z)=\sum^{\infty}_{k=0}b_{m+k}z^k
\begin{enumerate}
 \item $T_{\overline{z}^m + f}g=0;$
 \item $\big[ z g_1' -(m-1) g_1\big] + (m+1)z^m(fg)+z^{m+1}(fg)'=0;$
   \item $\big( z g_2' + g_2\big) + (m+1)(fg)+z(fg)'=0;$\label{equation}
 \item $d_{m+k}=-(m+1+k)/(k+1)\sum^k_{i=0}a_{k-i}d_i, k=0, 1, \ldots .$
\end{enumerate}
\end{lemma}
\begin{proof}
%---------------------(1)=(2)
In light of Lemma \ref{tool1}, the equation $T_{\overline{z}^m + f}g=0$ is equivalent to the equation $T^*_{z^m} g_1  + fg =0$. It follows that
$$\frac{1}{z^{m+1}}\int^z_0 \big[ t g_1'(t) -(m-1) g_1(t)\big] dt + fg =0.$$
Differentiating both sides of this equation, we have
\begin{equation}
\big[ z g_1' -(m-1) g_1\big] + (m+1)z^m(fg)+z^{m+1}(fg)'=0.
\end{equation}
 This proves $ (1)\Rightarrow (2)$.
%----------------(2)=(3)
Using the factorization $g_1(z) =z^m g_2(z)$ and part $(2)$, we get
$$\big[ z^{m+1} g_2'+ mz^mg_2 -(m-1)z^m g_2\big] + (m+1)z^m(fg)+z^{m+1}(fg)' =0,$$ which is equivalent to the equation
\begin{equation}\label{equation}
 \big( z g_2' + g_2\big) + (m+1)(fg)+z(fg)'=0.
\end{equation}
Hence we have $(2)\Leftrightarrow(3)$.
%--------------(3)=(4)
Furthermore, taking the $k$-th derivative of Equation \eqref{equation}, we arrive at
\begin{align}\label{equationk}
zg_2^{(k+1)}+(k+1)g^{(k)}_2+(m+1+k)(fg)^{(k)}+z(fg)^{(k+1)}=0.
\end{align}
Setting $z=0$ in Equation \eqref{equationk} and
using the expansion $$\big( f(z)g(z) \big)^{(k)} = \sum^k_{i=0}k! \frac{f^{(k-i)}(z)}{(k-i)!}\frac{g^{(i)}(z)}{i!},\ \ \ k \geq 1,$$ we obtain
 %$b_m=-(m+1)a_0b_0$, $2b_{m+1}=-(m+2)(a_0b_1+a_1b_0)$
\begin{align}\label{cofficient}
d_{m+k}=-\frac{m+1+k}{k+1} \sum^k_{i=0}a_{k-i}d_i, k \geq 0,
\end{align}
which shows $(3)\Rightarrow(4)$. The fact that $(4)\Rightarrow(1)$ is easily checked by direct computation, and this completes the proof of the lemma.
\end{proof}
%--------------------corollary ----
The following is an immediate consequence.
\begin{corollary}\label{cor1}
Suppose $g\in L^2_a(\D)$ is such that $g_0=0$. Then $T_{\overline{z}^m + f}g=0$ if and only if $g\equiv 0.$
\end{corollary}
%-----the proof of cor
\begin{proof}
In view of Lemma \ref{thm ker}, we see that $T_{\overline{z}^m + f}g = 0$ if and only if
\[d_{m+k}=-(m+1+k)/(k+1)\sum^k_{i=0}a_{k-i}d_i, k=0, 1, \ldots.\] Since $d_0= d_1= \cdots = d_{m-1}=0,$
we must have $d_m= d_{m+1} = \cdots = d_{2m-1}=0$. The corollary then follows from induction.
\end{proof}
%------- the kernel of polynomial symbol------
For the special case $f=cz^n$, where $n \geq 0$ and $c \in \mathbb{C}$, we have the following description of the kernel.
\begin{lemma}\label{poly ker}
The following holds for the Toeplitz operator $T_{\overline{z}^m + cz^n}$.
\begin{enumerate}
\item For $|c| \geq 1, \ker \left(T_{\overline{z}^m + cz^n}\right)= \{0\}$.
\item For $|c|<1$, $g\in \ker\left( T_{\overline{z}^m + cz^n}\right)$ if and only if it is of the form
\[g(z) = \sum^{m-1}_{j=0}d_j\bigg( 1+ \sum^{\infty}_{k=1} (-c)^k \bigg( \prod_{i=1}^k \frac{i(m+n)+1+j}{in+(i-1)m+1+j}   \bigg) z^{k(m+n)+j}\bigg). \]
\end{enumerate}
\end{lemma}
%-------- the proof of the polynomial  symbol kernel----------------
\begin{proof}
In light of the fourth statement of Lemma \ref{thm ker}, $T_{\overline{z}^m + f}g=0$ holds if and only if $$d_{m+k}=-(m+1+k)/(k+1)\sum^k_{i=0}a_{k-i}d_i, k=0, 1, \ldots. $$
\noindent Since $a_n=c$ and $ a_k= 0,k \neq n,$ we have $d_m= d_{m+1} = \cdots = d_{m+n-1}=0$, and
\[d_{m+n}=-(m+1+n)/(n+1)a_nd_0=(-c)(m+1+n)/(n+1)d_0.\] By induction,
\begin{align*}
d_{k(m+n)+j}&=(-c)\frac{k(m+n)+1+j}{kn+(k-1)m+1+j}d_{(k-1)(m+n)+j},\\
&=(-c)^k \prod_{i=1}^k \frac{i(m+n)+1+j}{in+(i-1)m+1+j},\ \ \ 0\leq j\leq(m-1), k \geq 1.
\end{align*}
Regrouping the terms in the power series expansion for $g$, we can write
\begin{align}
g(z) &= \sum^{m-1}_{j=0}d_j\bigg( 1+ \sum^{\infty}_{k=1} (-c)^k \bigg( \prod_{i=1}^k \frac{i(m+n)+1+j}{in+(i-1)m+1+j}   \bigg) z^{k(m+n)+j}\bigg)\\
&=: \sum^{m-1}_{j=0}d_j (1+ z^j\widehat{g}_j(z)).
\end{align}
Thus, function $g\in  L^2_a({\mathbb D})$ if and only if $ \widehat{g}_j\in L^2_a({\mathbb D})$ for each $j$.
Direct calculation shows that
\begin{align}
\|\widehat{g}_j\|^2 = \sum^{\infty}_{k=1} |c|^{2k} \prod_{i=1}^k \frac{[i(m+n)+1+j]^2}{[in+(i-1)m+1+j]^2}\frac{1}{k(m+n)+1},
\end{align}
which is convergent for  $|c|<1$. When $|c| \geq 1$, we have $$\sum^{\infty}_{k=1} \prod_{i=1}^k \frac{[i(m+n)+1+j]^2}{[in+(i-1)m+1+j]^2}\frac{1}{k(m+n)+1} \geq \sum^{\infty}_{k=1} \frac{1}{k(m+n+1)}=\infty,$$ it implies that $\widehat{g}_j\notin L^2_a({\mathbb D})$. This completes the proof of the lemma.
\end{proof}
%------------------------the dimension of the polynomial symbol kernel -----------
It is well-known that $T_{{\bar z}^m}=(T^*_z)^m$ is of Coburn type with $\dim \ker\left(T_{{\bar z}^m}\right)=m$ and $\dim \ker\left(T_{z^m}\right)=0$. For $c\neq 0$, we have $(T_{\overline{z}^m + cz^n})^*=\bar{c}T_{\overline{z}^n + z^m/\bar{c}}$. The following direct consequence of Lemma \ref{poly ker} is the main result of this section.
\begin{theorem}\label{cor3}
$T_{\overline{z}^m + cz^n}$ is of  Coburn~ type for each $c\in \C$ and $m \geq 1$, $n \geq 0$. Moreover, either
\begin{enumerate}
\item $\dim \ker\left(T_{\overline{z}^m + cz^n} \right)= m$ and $\dim \ker \left(T_{\overline{z}^m + cz^n}^* \right)= 0$, which occurs when $|c|<1$, or
\item $\dim \ker \left(T_{\overline{z}^m + cz^n}\right)= 0$ and $\dim \ker\left( T_{\overline{z}^m + cz^n}^*\right)= n$, which occurs when $|c|\geq 1$.
\end{enumerate}
\end{theorem}

The following lemma characterizes the range of $T_{\overline{z}^m + f}$.
%---------------------range----------
\begin{lemma}\label{thm ran}%-----range---
Suppose $h(z)=\sum^{\infty}_{k=0}c_kz^k \in L^2_a({\mathbb D})$. Then the  following conditions are equivalent:
%----------u_2(z)=\sum^{\infty}_{k=0}b_{m+k}z^k
\begin{enumerate}
 \item $T_{\overline{z}^m + f}g=h;$
 \item $\big[ z g_1' -(m-1) g_1\big] + (m+1)z^m(fg-h)+z^{m+1}(fg-h)'=0;$
   \item $\big( z g_2' + g_2\big) + (m+1)(fg)+z(fg)'= (m+1)h + zh';$
    \item $c_k = (k+1)/(m+k+1)d_{m+k}+\sum^k_{i=0}a_{k-i}d_i, k=0, 1, \ldots.$
\end{enumerate}
\end{lemma}
\begin{proof}%----ran
It is sufficient to show that (3) implies (4), as the rest of the proof is similar to that for Lemma \ref{thm ker}.
In fact, taking the $k$-th $(k \geq 0)$ derivative of statement (3), we have that
\begin{align}\label{equationk}
zg_2^{(k+1)}(z) +(k+1)g^{(k)}_2 (z) &+ (m+1+k)(fg)^{(k)} (z) + z(fg)^{(k+1)} (z)\notag\\
& =(m+k+1)h^{(k)}(z) + zh^{(k+1)}(z).
\end{align}
Setting $z=0$, we arrive at
\begin{align}\label{ran cofficient}
c_{k}=\frac{k+1} {m+1+k}d_{m+k}+\sum^k_{i=0}a_{k-i}d_i, k \geq 0,
\end{align}
which is (4).
\end{proof}
The next corollary is immediate.
%----------------------- range of polynomials symbols
\begin{corollary}\label{poly ran}
For any $n \geq 1, c \in \mathbb{C}$, it holds that $T_{\overline{z}^m + cz^n}g = \sum^{\infty}_{j=0}c_jz^j$, where
 \begin{eqnarray*}
c_j =   \left\{
                          \begin{aligned}
                           \frac{j+1}{m+j+1}d_{m+j}\ \ &  0 \leq j \leq n-1,   \\
                           \frac{j+1}{m+j+1}d_{m+j} + c d_{j-n}  \ \ & j \geq n.
                          \end{aligned}
                           \right.
\end{eqnarray*}
\end{corollary}

%------------------------------------------------------------
%          The kernels of $T_{\overline{q} + p}$ and the connectedness of $\sigma (T_{\overline{q} + p})$
%-------------------------------------------
\section{The kernels of $T_{\overline{q} + p}$ and the connectedness of  $\sigma~(T_{\overline{q} + p})$}
By using the methods of difference equations and differential equations, the kernels   and the spectra of  certain Toeplitz operators on $L^2_a({\mathbb D})$ are  characterized.
We first introduce the following theorem, which describes injective  Toeplitz operators with harmonic polynomial symbols of the form $\varphi  = \overline{q}+ p$, where $p$ and $q$ are two analytic polynomials.

For convenience, we write \[q(z) = z^m+\sum_{i=1}^{m-1}\overline{\alpha}_{-(m-i)}z^{m-i},\ \ \ p(z)=\sum_{i=0}^n \alpha_iz^i~(\alpha_n~\neq 0),\] throughout this section.
We say that a polynomial satisfies \emph{Poincar\'{e} condition}  if its zeros have distinct moduli.

%----------------------------------------injective Toeplitz ope
\begin{theorem}\label{injective}
Assume    $\varphi (z) = \overline{q(z)}+ p(z)$.
Suppose that the associated polynomial 
\[\varphi_0 (z):=1+\sum_{i=1}^{m+n}\alpha_{i-m}z^{i} \] 
satisfies Poincar$\acute{e}$ condition and exists at least one zero  inside $\mathbb {D}$.
Then  $ \mathrm{ker}(T_{\varphi}) = \{0\}.$
\end{theorem}
%---------  proof
\begin{proof}
For any $g(z)=\sum^\infty_{k= 0}d_k z^k\in L^2_a(\mathbb{D})$, with application of  Lemma \ref{p} we have
\[ T_{\overline{z}^i}g(z)=\sum^\infty_{k= 0} \frac{k+1}{i+k+1} d_{i+k} z^k~(i \geq 1),\]
and hence
\[ T_{\overline{q}} g(z)=\sum^\infty_{k= 0} \left( \frac{k+1}{m+k+1} d_{m+k}+\sum_{i=1}^{m-1}\alpha_{i-m} \frac{k+1}{m-i+k+1}d_{m-i+k} \right)z^k .\]
Therefore, $T_\varphi g=0$ if and only if
\begin{equation}\label{pq}
\frac{k+1}{m+k+1} d_{m+k}+\sum_{i=1}^{m-1}\alpha_{i-m} \frac{k+1}{m-i+k+1}d_{m-i+k}+\sum_{i=0}^n \alpha_i d_{k-i}=0,\quad k\geq 0,
\end{equation}
assuming $d_j=0$ if $j<0$. The corresponding characteristic equation is
\begin{equation}\label{2.2}
 0=z^{m+n}+\sum_{i=1}^{m+n} \alpha_{i-m} z^{m+n-i}=z^{m+n}\varphi_0(1/z).
\end{equation}
Since $\varphi_0$ is assumed to be Poincar\'{e}, the polynomial $z^{m+n}\varphi_0(\frac{1}{z})$ also has zeros of distinct moduli, say $\lambda_1,\lambda_2,\cdots,\lambda_{m+n}$ with the order $|\lambda_1|> \cdots >|\lambda_{n+m}|$. Then, we must have $|\lambda_1|> 1$. Furthermore, we have
\[ (-1)^{i}\alpha_{i-m}=s_{i}(\lambda_1,\lambda_2,\cdots, \lambda_{m+n}),\ \ \ i=1,\cdots,m+n, \]
where $s_i$ denote the $i$-th fundamental symmetric polynomial of $\lambda_1,\lambda_2,\cdots,\lambda_{m+n}$.

If $d_0=d_1=\cdots=d_{m-1}=0$, we have $g(z)=0$ using  recursion relation \eqref{pq}. The proof is  completed.
If $(d_0, d_1,\cdots, d_{m-1})$ is nonzero, by recursion relation \eqref{pq}, we conclude that $d_{m+k}~( k \ge 0)$ is a symmetric polynomial of $\lambda_1, \lambda_2, \cdots, \lambda_{m+n}$ of degree at most $m+k$.
Thus, if the characteristic equation \eqref{2.2} satisfies Poincar\'{e} condition, then Poincar\'{e} Theorem  \ref{thm1.1} implies that $\lim_{k\to \infty}d_{k+1}/d_{k}=\lambda_j$  for some $j$. Since $d_{k}$  is a symmetric polynomial of $\lambda_1, \lambda_2, \cdots, \lambda_{m+n}$ of degree $k$, then by Perron's Theorem \ref{thm1.3}, $\lambda_i$ must be the eigenvalue of the maximum modulus, i.e., $\lambda_1$.  But this implies that $g(z) \notin Hol({\mathbb D})$ as $|\lambda_1|>1$, this is a contradiction. And so we must have $\ker\left( T_\varphi\right) =\{0\}$.
This completes the proof.
\end{proof}

%----------

When $T_\varphi$ with $\varphi (z)= \overline{z}^m + \alpha z^m + \beta$, we fully characterize the  kernel of Toeplitz operator $T_\varphi$.

{\bf Case 1}~$\alpha =0$.

See Lemma \ref{poly ker}(2) with $\beta =c$, which is equivalent to that  the zeros of $1 + \beta z^m$ being outside $\D$. Then $\dim \ker \left(T_\varphi\right) = m$.

{\bf Case 2}~$\alpha \neq 0$.

{\bf Subcase 1}~$\beta^2 \neq 4\alpha^2.$

In this case, the  zeros $z^m_0$ and $z^m_1$ of $\alpha z^{2m}+ \beta z^m + 1$ are not equal.

For $g(z):= \eta(z)+\eta_0(z)=\sum^\infty_{k=0}d_k z^k \in L^2_a({\mathbb D}) $ such that
\[ T_{\varphi}g (z)= T_{\overline{z}^m}\eta(z) +  T_{ \alpha z^m + \beta}[\eta(z)+\eta_0(z)] =0,\]
where $\eta_0$ is the polynomial which degree  is $(m-1)$ at most and $\eta(z) \in L^2_a({\mathbb D}) $  satisfying that $0$ is the zero with  $m$ multiple at least.
Using Lemma \ref{tool1}, we have the following equation
\begin{align*}\label{eq}
& \big[ m(2\alpha z^{2m}+ \beta z^m -1)+ (\alpha z^{2m}+ \beta z^m + 1) \big] \eta(z) + z \big( \alpha z^{2m}+ \beta z^m + 1  \big) \eta'(z)\notag \\
&\ \ \ \  \ \ \ \  \ \ \ \  \ + \big[ m \alpha z^{2m} + (m+1)(\alpha z^{2m} + \beta z^m )\big] \eta_0(z) + z( \alpha z^{2m} +\beta z^m ) \eta'_0(z)=0.
\end{align*}
With the notation of $\delta (z):=(m-1)\eta_0(z)-z\eta'_0(z)= d_{m-2}z^{m-2}+ 2d_{m-3}z^{m-3}+\cdots+(m-2)d_1z+(m-1)d_0$~($m=1, \delta (z)\equiv 0$), which is a polynomial with degree is  $m-2$ at most, where $\eta_0(z)= d_0+ \cdots + d_{m-1}z^{m-1}$,
the above equation can be rewritten as
\begin{align}
 \big[  m(2\alpha z^{2m}+ \beta z^m -1)&+ (\alpha z^{2m}+ \beta z^m + 1) \big] g(z)\notag\\
  &+ z \big( \alpha z^{2m}+ \beta z^m + 1  \big) g'(z) + \delta (z)=0.
\end{align}
Set $G(z) :=  z(\alpha z^{2m}+ \beta z^m + 1) g(z).$
Since $g(z) \in L^2_a({\mathbb D})$, we have  $G(z) \in L^2_a({\mathbb D})$.
Then the above equation is
\begin{eqnarray}\label{inte}
G'(z)-\frac{m}{z(\alpha z^{2m}+ \beta z^m + 1)} G(z)+\delta (z)=0.
\end{eqnarray}
With the assumption that $\alpha z^{2m}+ \beta z^m + 1= \alpha (z^m-z^m_0)(z^m-z^m_1)$ and $|z_0| > 1$ and  $|z_1| > 1$. Set
\begin{eqnarray}\label{G0}
G_0 (z):= z^m    \frac{(z^m-z^m_0)^{\frac{z^m_1}{z_0^m-z^m_1}}}{(z^m-z^m_1)^{\frac{z^m_0}{z_0^m-z^m_1}}},
          \end{eqnarray}
where $(z^m-z^m_0)^{\frac{z^m_1}{z_0^m-z^m_1}}= e^{\frac{z^m_1}{z_0^m-z^m_1} \mathrm{ln}(z^m-z^m_0)}$.

By some direct calculations, we have
\begin{align*}
\frac{G'_0 (z)}{G_0 (z)}&=\frac{m}{z}+\frac{z^m_1}{z_0^m-z^m_1}\frac{mz^{m-1}}{z^m-z^m_0}-\frac{z^m_0}{z_0^m-z^m_1}\frac{mz^{m-1}}{z^m-z^m_1}\\
                       &=\frac{m}{z}+ mz^{m-1}\frac{-(z_0^m-z^m_1)z^m+(z_0^{2m}-z^{2m}_1)}{(z_0^m-z^m_1)(z^m-z^m_0)(z^m-z^m_1)}\\
                         &=\frac{m}{z}+ mz^{m-1}\frac{-\alpha z^m+\alpha(z_0^m+z^m_1)}{\alpha (z^m-z^m_0)(z^m-z^m_1)}\\
                           &=\frac{m}{z}- mz^{m-1}\frac{\alpha z^m+\beta}{\alpha (z^m-z^m_0)(z^m-z^m_1)}~(\text{using} ~z_0^m+z^m_1 =-\frac{\beta}{\alpha} )\\
                           &=\frac{m}{z}- m\frac{\alpha z^{2m-1}+\beta z^{m-1}}{\alpha (z^m-z^m_0)(z^m-z^m_1)}\\
                           &=\frac{m}{z(\alpha z^{2m}+ \beta z^m + 1)}.~\left(\text{using} ~\alpha z^{2m}+ \beta z^m + 1= \alpha (z^m-z^m_0)(z^m-z^m_1) \right)
\end{align*}
Consider the following function
\begin{eqnarray}
G(z) :=  G_0(z) \left(  -  \int^z_{0} \frac{\delta (t)}{G_0(t)}dt +c  \right),  \label{Gz}
\end{eqnarray}
where $z \in \mathbb {D}$ and $c$ is  a constant.
Since $\lim\limits_{z \rightarrow 0} z^m \int^z_0 t^{-m} dt  = 0, $ we conclude that $G(z)$ is a bounded function.
Since  $G'(z) =  G'_0(z) \left(  -  \int^z_{0} \frac{\delta (t)}{G_0(t)}dt +c  \right) -\delta (z),$ we have
\begin{align*}
G'(z)&-\frac{m}{z(\alpha z^{2m}+ \beta z^m + 1)} G(z)+\delta (z)\\
   &= G'_0(z) \left(  -  \int^z_{0} \frac{\delta (t)}{G_0(t)}dt +c  \right) -\delta (z) -\frac{m}{z(\alpha z^{2m}+ \beta z^m + 1)} G(z)+\delta (z)\\
                                                               &=G'_0(z) \left(  -  \int^z_{0} \frac{\delta (t)}{G_0(t)}dt +c  \right) -\delta (z)-  \frac{G'_0 (z)}{G_0 (z)} G(z)+\delta (z)\\
                                                                &=G'_0(z) \left(  -  \int^z_{0} \frac{\delta (t)}{G_0(t)}dt +c  \right)-G'_0(z) \left(  -  \int^z_{0} \frac{\delta (t)}{G_0(t)}dt +c  \right)\\
                                                                &=0.
\end{align*}
So $G(z)$ is the solution  of Equation  \eqref{inte}.
Hence
\begin{align}\label{root}
g(z)&= \frac{G(z)}{z(\alpha z^{2m}+ \beta z^m + 1)} = \frac{G_0(z)}{z\alpha (z^m-z^m_0)(z^m-z^m_1)}\left(  -  \int^z_{0} \frac{\delta (t)}{G_0(t)}dt +c  \right)\notag\\
    &= z^{m-1}\frac{(z^m-z^m_0)^{\frac{2z^m_1}{z_0^m-z^m_1}}}{\alpha(z^m-z^m_1)^{\frac{2z^m_0}{z_0^m-z^m_1}}}\left(  -  \int^z_{0} \frac{\delta (t)}{G_0(t)}dt +c  \right).
\end{align}
If $m=1$, then $\delta (z)=0$; if $m>1$, then $\lim\limits_{z \rightarrow 0} z^{m-1} \int^z_0 t^{-m} dt $ exist.
So $g(z)$ is a bounded analytic function.
As is known to all, the solution space of  a first-order linear differential equation is at most 1-dimensional and by the definition of $\delta (z)= d_{m-2}z^{m-2}+ 2d_{m-3}z^{m-3}+\cdots+(m-2)d_1z+(m-1)d_0$ and the fact that $\delta (z)\equiv 0$ perhaps  occur, we have $\mathrm{dim~ker}(T_\varphi ) = m$.

By Vieta's theorem,  $\frac{1}{z^m_0}$ and $\frac{1}{z^m_1}$ are distinct zeros of  the equation $z^{2m}+ \beta z^m +\alpha=0$ inside $\D$.
Applying Theorem \ref{dis2},  $\alpha z^{2m}+ \beta z^m +\alpha=0 $ has all its zeros  in $\mathbb {D}$
if and only if $1-|\alpha|^2 >0$ and $1-|\alpha|^2 > |\alpha \overline{\beta} - \beta|$,
 that  is  $1-|\alpha|^2 >|\alpha \overline{\beta} - \beta|$.
This leads to the following result.
\begin{theorem}
Suppose that $\varphi (z)= \overline{z}^m + \alpha z^m + \beta$, $ \alpha z^{2m} + \beta z^m + 1= \alpha(z^m-z^m_0)(z^m-z^m_1)$ ~with  $1 > 1-|\alpha|^2>|\alpha \overline{\beta} - \beta|$ and $ \beta^2 \neq 4 \alpha$.
Then \[ \ker\left(T_{\varphi}\right)= \mathrm{Span}\{ g_1, \cdots, g_m\},\]
~where
$g_1 (z)= \frac{G_0 (z)}{\alpha~ z(z^m-z^m_0)(z^m-z^m_1)}$,~$g_j (z)= - \frac{G_0 (z)}{\alpha~ z(z^m-z^m_0)(z^m-z^m_1)}\int^z_0\frac{t^{j-2}}{G_0 (t)}dt~(j \ge 2)$ are linearly independent, the definition of   $G_0(z)$ is given in \eqref{G0}.
\end{theorem}

{\bf Subcase 2}~$\beta^2 = 4\alpha.$

Set
\begin{eqnarray}\label{expre}
 \alpha z^{2m}+ \beta z^m + 1 = \alpha (z^m -z_0^m)^2~ \text{with}~z_0^m= -\frac{\beta}{2 \alpha}
\end{eqnarray}
and
\begin{eqnarray}\label{root2}
G_1(z)=\frac{ z^{m-1}}{\big(\alpha z^{2m}+ \beta z^m + 1\big)^{\frac{3}{2}}} e^{\frac{\beta}{2\alpha } \frac{1}{z^m -z_0^m }}.
\end{eqnarray}
Replacing
\begin{eqnarray}\label{root3}
g(z)=G_1(z) \left(  -  \int^z_{0} \frac{\delta (t)}{t\big(\alpha t^{2m}+ \beta t^m + 1\big)G_1(t)}dt +c  \right)
\end{eqnarray}
with $g(z)$ in Equation \eqref{root}, and similar  to the proof of the above theorem, we obtain the following theorem.
\begin{theorem}
Suppose that $\varphi (z)= \overline{z}^m + \alpha z^m + \beta$~with $ \beta^2 = 4 \alpha$ and $0 < |\beta| < 2.$
Then \[ \ker\left(T_{\varphi}\right)= \mathrm{Span}\{ h_1, \cdots, h_m\},\]
~where
$h_1 (z)= G_1(z)$,~$h_j (z)= -G_1(z) \int^z_0\frac{t^{j-2}}{t\big(\alpha t^{2m}+ \beta t^m + 1\big)G_1 (t)}dt~(j \ge 2)$ are linearly independent, the definition of   $G_1(z)$ is given in \eqref{root2}.
\end{theorem}
The proofs of the above two theorems and Lemma \ref{thm ker} stimulate us to raise the following problem.
\begin{problem}
Is $\mathrm{dim~ker}\left(T_{\overline{z}^m +f}\right) = m$ for a bounded function $f$ satisfying that all the zeros of $ 1 +zf$  are outside $\mathbb D$.
\end{problem}

Let $\varphi (z) = \overline{q(z)}+ p(z)$ as in Theorem \ref{injective}.
 For any $\lb\in \C$, we define an associated polynomial \[ \varphi_\lb(z):=\left(1+\sum_{i=1}^{m+n}\alpha_{i-m}z^{i}\right)-\lb z^m.\] Observe that $\varphi_\lb(0)=1, $ and on $ \mathbb {T}$ we have $\varphi(z)-\lb=\frac{\varphi_\lb(z)}{z^m}$.
The following theorem is our main result for this section.
\begin{theorem}\label{main3}
If $\varphi_\lb$ is Poincar\'{e} for every $\lb \notin \varphi(\mathbb T)$, then
\begin{eqnarray*}
 \sigma (T_{\varphi})&=&\varphi (\mathbb T ) \bigcup \big\{ \lambda \in \mathbb {C}: \lambda \notin   \sigma_e(T_{\varphi}), \ind (T_{\varphi} - \lambda I ) \neq 0 \big  \}\\
                      &=& \varphi (\mathbb T ) \bigcup \big\{ \lambda \in \mathbb {C}: \lambda \notin   \varphi (\mathbb T), \wind( \varphi (\mathbb T), \lambda ) \neq 0 \big  \}.
\end{eqnarray*}
\end{theorem}
%---------  proof
\begin{proof} In light of Theorem \ref{specT}, it only remains to show that \[\sigma_p(T_\varphi)\bigcap \{\lambda\in \mathbb{C}: \lambda \notin \sigma_e(T_\varphi), \ind(T_\varphi-\lambda I)=0\}=\emptyset.\]
In fact, we will show that if $T_\varphi - \lambda$ is Fredholm  with $\text{ind}(T_\varphi -\lambda )=0$, then it is invertible. First, we observe that
\begin{align}
\varphi(z) - \lambda &= \frac{z^m \overline{q(z)}+ z^mp(z) - \lambda z^m}{z^m} \notag \\
                     &=\frac{1}{z^m}\left(1+\sum_{i=1, i \neq m}^{m+n}\alpha_{i-m}~z^{i}+ (a_0 -\lambda )z^m \right)=\frac{\varphi_\lambda}{z^m},\ \  z \in \mathbb T.
\end{align}
Without loss of generality, we prove the claim for the case $\lambda =0$, i.e., we show that $0\notin \sigma(T_\varphi)$.
Using Theorem \ref{injective}, we have $\ker\left(T_\varphi \right) =\{0\}$. Furthermore, since $T_\varphi$ is Fredholm with index $0$, it must be invertible. This completes the proof.
\end{proof}
Theorem \ref{main3} indicates that $\sigma(T_\varphi)$ in this case is the region enclosed by the closed path $\varphi(\T)$. Hence,
it has the following direct consequence.
\begin{corollary} Let $\varphi$ be as in Theorem \ref{main3}. Then $\sigma(T_{\varphi})$ is path-connected.\end{corollary}
\noindent This result is reminiscent of a fact in the Hardy space over the unit disc, where Widom's theorem \cite{Doug,Widom1} asserts that every Toeplitz operator with essentially bounded symbol has a path-connected spectrum.
%-------------------------------corollary
\begin{corollary}
Let $\varphi$ be as in Theorem \ref{main3} and assume the associated polynomial $\varphi_0$ is Poincar\'{e}. Then the following are equivalent.
\begin{enumerate}
\item $T_\varphi$ is invertible.
\item $T_\varphi$ is Fredholm with index $0$.
\item $\varphi_0$ has $m$ zeros in $\D$ and no zeros in $ \T$.
\end{enumerate}
\end{corollary}

Although the condition in Theorem \ref{main3} that  $\varphi_\lb$ is Poincar\'{e} for every $\lb \notin \varphi(\mathbb T)$ may appear rather restrictive, it is in fact indispensable.~Counterexamples are known to exist for the case $\varphi=\overline{z}+p$, where $p$ is a polynomial with degree $k>2$ \cite[Theorem 4.1]{Invetible2020}.
Nevertheless, it is possible to weaken this condition in some way. Further study is needed in this direction.

%---------------------------------------------------------------
% The spectrum of $T_{\overline{z}^m + \alpha z^m + \beta}$
%-----------------------------------------------------------------------------------------------
\section{The Spectral Properties of $T_{\overline{z}^m + \alpha z^m + \beta}$ }

\subsection{The spectrum of $T_{\overline{z}^m + \alpha z^m + \beta}$ }

Lemmas \ref{thm ker} and \ref{thm ran} establish a connection between Toeplitz operator theory on the Bergman space and linear differential equation theory. In this section, Poincar\'{e}'s theorem  will be used to describe the spectrum of $T_\varphi$, where $\varphi (z)$ is of the form $\overline{z}^m + \alpha z^m + \beta$ with $m\geq 1$ and $\alpha, \beta \in \mathbb {C}$. In the case $m=1$, Zhao and Zheng \cite[Theorem 3.1]{spec2016} shows that $\sigma(T_\varphi)=\overline{\varphi (\mathbb D)}$. This section extend this result to the case $m\geq 1$.

%----------------------------------------$\overline{\varphi (\mathbb {D}) }\subset \overline{\sigma(T_\varphi)} = \sigma(T_\varphi).$
\begin{proposition}\label{inclusion}
Assume $\varphi (z)= \overline{z}^m + \alpha z^m + \beta$ with $\alpha, \beta \in \mathbb {C}, m \geq 1$.  Then $\overline{\varphi (\mathbb {D})} \subseteq \sigma (T_\varphi)$.
\end{proposition}
\begin{proof}%----proof
For the case $\alpha =1$, refer to \cite{MC1979}.
Without loss of generality, set $\alpha >0$ and  $\alpha \neq 1$. For more details, see \cite{spec2016}.
Set $z=re^{\text{i}\theta} \in \mathbb {D}$, $ \varphi (z) = x+\text{i}y$ and $\beta= \beta_1 + \text{i} \beta_2$.
 We have
\begin{align}
 x &=( \alpha + 1)  r^m \cos (m\theta) + \beta_1,\\
 y & = ( \alpha -1 ) r^m \sin (m\theta) + \beta_2.
\end{align}
It follows that
\begin{eqnarray}
 \frac{(x-\beta_1)^2}{( \alpha + 1)^2 }  + \frac{(y-\beta_2)^2}{( \alpha - 1)^2  } = r^{2m} < 1.
\end{eqnarray}
In light of Theorem \ref{specT}, we have
\iffalse
\begin{align*}
 \sigma (T_{\varphi}) &= \varphi (\mathbb T ) \bigcup \big\{ \lambda \in \mathbb {C}: \lambda \notin   \varphi (\mathbb T), \text{wind}( \varphi (\mathbb T), \lambda ) \neq 0 \big  \}\\
 & \bigcup \big\{ \lambda \in \mathbb {C}:   \lambda \in  \sigma_p(T_{\varphi}),  \lambda \notin  \sigma_e(T_{\varphi}), \text{ind} (T_{\varphi}- \lambda I)=0 \big\},
\end{align*}
 we have
 \fi
 \begin{align*}
 \varphi ( \mathbb {D}) =\{ x+\text{i}y:  \frac{(x-\beta_1)^2}{( \alpha + 1)^2 }  + \frac{(y-\beta_2)^2}{( \alpha - 1)^2  } < 1 \} \subseteq  \sigma (T_{\varphi}).
\end{align*}
It follows that $\overline{\varphi ( \mathbb {D})}\subseteq \sigma (T_{\varphi})$, and this completes the proof.
\end{proof}
%--------------------------- \sigma (T_\varphi) \subset \overline{\varphi (\mathbb {D})}
%
\begin{remark}\label{remark4.2}
It is worth noting that the inclusion $\overline{\varphi (\mathbb {D})} \subseteq \sigma (T_\varphi)$ does not hold for a general harmonic symbol $\varphi$. For example, it is shown \cite[Theorem 4.1]{spec2016} that for $\varphi (z)=\bar{z}+z^2-z$, one has
$[0, 1)\subseteq \overline{\varphi ({\D})}$ but $[0, 1)\cap \sigma(T_\varphi)=\emptyset$.
\end{remark}

%----------
The following is the main result of this section.
\begin{theorem}\label{main m}
Assume $\varphi (z)= \overline{z}^m + \alpha z^m + \beta$ with $m\geq 1, \alpha, \beta \in \mathbb {C}$.  Then $ \sigma (T_\varphi) = \overline{\varphi (\mathbb {D})}.$
\end{theorem}
\begin{proof}%----proof
Given Proposition \ref{inclusion}, we only need to prove $ \sigma (T_\varphi) \subseteq \overline{\varphi (\mathbb {D})}$, or equivalently, if $w\notin \overline{\varphi (\mathbb {D})}$, then $T_\varphi-w$ is invertible. Without loss of generality, we assume $w=0$, i.e., $\varphi(z)$ does not vanish in $\overline{\D}$. Then, Lemma \ref{spec} indicates that $T_\varphi$ is a Fredholm operator with index
\begin{eqnarray}\label{index}
\text{ind}(T_\varphi )= - \frac{1}{2\pi \text{i}}\int_{\varphi (\mathbb T)} \frac{1}{w} dw=0.
\end{eqnarray}
In particular, the range of $T_\varphi$ is closed. Notice that
\begin{eqnarray*}
\varphi (z)= \overline{z}^m + \alpha z^m + \beta = \frac{\alpha z^{2m}+ \beta z^m +1 }{z^m},~ z \in \mathbb T.
\end{eqnarray*}
The Argument Principle and  Identity (\ref{index}) then indicate that
\begin{eqnarray}\label{spec equ1}
\alpha z^{2m}+ \beta z^m +1=0
\end{eqnarray}
has exactly $m$ simple zeros $z_0 \varpi^k \neq 0$ inside $\mathbb {D}$ and $m$ simple zeros $z_1\varpi^k $ outside $\mathbb {D}$, where $\varpi = e^{i\frac{2\pi}{m}}$ and $0\le k \le m-1$.

%-------------------------
If there exists $g(z)= \sum_{k=0}^\infty d_kz^k \in L^2_a({\mathbb D})$ such that $T_\varphi g =0$, then using Theorem \ref{p}, we have
\begin{align*}
0&= T_\varphi g(z)= T_{\overline{z}^m} g (z)+ \big(\alpha z^m + \beta\big)g(z)\\
                & = \sum_{k=0}^{m-1} \Big( \frac{k+1}{k+m+1}d_{k+m}+ \beta d_k \Big)z^k+ \sum_{k=m}^\infty \Big(   \frac{k+1}{k+m+1}d_{k+m}+ \alpha d_{k-m} + \beta d_k  \Big)z^k,
\end{align*}
 which implies
 \begin{eqnarray*}
  \left\{ \begin{aligned}
                           & \frac{k+1}{k+m+1}d_{k+m}+ \beta d_k =0, \ \ & & 0 \leq k\leq m-1, \\
                           & \frac{k+1}{k+m+1}d_{k+m}+ \alpha d_{k-m} + \beta d_k =0,\ \ & & k \geq m.
                          \end{aligned}
                           \right.
 \end{eqnarray*}
Setting $b_k = \frac{d_k}{k+1}$ for $k \geq 0$, we have
\begin{eqnarray}\label{bkm}
  \left\{ \begin{aligned}
                           & b_{k+m}+ \beta b_k =0, \ \ & & 0 \leq k\leq m-1, \\
                           & b_{k+m}+ \beta b_k+\alpha\frac{k-m+1}{k+1} b_{k-m}  =0,\ \ & & k \geq m.
                          \end{aligned}
                           \right.
 \end{eqnarray}
Since $g\neq 0$ by assumption, the recursion relation above indicates that $b_{k_0} \neq 0$ for some $0 \leq k_0 \leq m-1$.
For simplicity, we assume that $g$ is rescaled such that $b_{k_0}=1$.
So $b_{k_0+m}+ \beta =0~(\beta \neq 0)$ and
\begin{eqnarray}
b_{k+m}+ \beta b_k + \alpha\frac{k-m+1}{k+1} b_{k-m}  =0, \ \ k \geq m.
 \end{eqnarray}
Set $\lambda_0 = \frac{1}{z^m_0}, \lambda_1= \frac{1}{z^m_1}$ and observe  that they are the two roots of the equation $z^2+\beta z+\alpha=0$. Hence, $\lambda_0\lambda_1=\alpha,\lambda_0+\lambda_1=-\beta$, and it follows from (\ref{bkm}) that
\begin{align}\label{bkm2}
 b_{k_0+(n+1)m}&-(\lambda_0+\lambda_1) b_{k_0+nm} \notag \\
&+\lambda_0\lambda_1 \frac{(n-1)m+1}{nm+1}b_{k_0+(n-1)m}=0\ ( n \ge 1)
\end{align}
with initial condition $b_{k_0}=1, b_{k_0+m}=\lambda_0+\lambda_1$. This is a case of  Difference Equation (\ref{1.0}), where $d=2, t=n=1, 2, ...$, and it is Poincar\'{e} with
$c_0=\lambda_0\lambda_1, c_1=-(\lambda_0+\lambda_1)$. Poincar\'{e}'s Theorem then implies that
\begin{equation}\label{poincare}
\lim_{n\to \infty}\frac{b_{k_0+(n+1)m}}{b_{k_0+nm}}= \lambda_i,\ \ \ i=0\ \text{or}\ 1.
\end{equation}
Since the initial conditions and the recursion (\ref{bkm2}) are both symmetric in $\lambda_0$ and $\lambda_1$, the term $b_{k_0+nm}$ must be a symmetric polynomial in $\lambda_0, \lambda_1$ of degree $n$. We assume $|\lambda_0|> 1 > |\lambda_1|$. If the leading term of $b_{k_0+nm}$, as a polynomial in $\lambda_0$, is denoted by $c(n)\lambda_0^{k(n)}\lambda_1^{n-k(n)}$ for some integer $1\leq k(n)\leq n$, then this term is the dominating term as $n\to \infty$. It follows that the limit in (\ref{poincare}) must be $\lambda_0$. But this implies that $g(z) \notin L^2_a({\mathbb D})$, which is a contradiction. Thus we get $\ker \left(T_{\varphi} \right)=\{0\}$. Since $T_\varphi$ is Fredholm with index $0$, we must also have $\ker\left(T^*_{\varphi}\right)  =\{0\}$, and hence $T_\varphi$ is invertible. This completes the proof.
\end{proof}

Remark \ref{remark4.2} and Theorem \ref{main m} motivate the following natural problem.
%----------------
\begin{problem}
Characterize harmonic symbols $\varphi$ for which $\sigma (T_\varphi) = \overline{\varphi (\mathbb {D})}$.
\end{problem}
%--------------------
%
%---------------------------------------

\subsection{Essential projective spectrum and Fredholm index}

To further analyze the spectral property of the Toeplitz operator $T_\varphi$, we recall that the projective spectrum of elements $A_1, ..., A_n$ in a Banach algebra ${\mathcal B}$ is defined as
\[ P(A)=\{z\in \C^n:\ A(z):=z_1A_1+\cdots +z_nA_n\ \text{is not invertible}\},\] and it has been investigated extensively since 2009. For details, we refer the reader to \cite{Ya24}.

Consider an infinite dimensional Hilbert space $\mathcal H$ and let $\mathcal{B}(\mathcal H)$ ($\mathcal{K}(\mathcal H)$) denote the $C^*$-algebra of bounded (resp. compact) linear operators on ${\mathcal H}$. The following definition thus makes a good sense.
\begin{definition}
For an infinite dimensional Hilbert space $\mathcal H$, the essential projective spectrum of linear operators $A_1, ..., A_n$ in $\mathcal{B}(\mathcal H)$ is defined as
\[P_e(A)=\{z\in \C^n:~ A(z)~ \text{is not Fredholm}\}.\]
\end{definition}
It is well-known that ${\mathcal K}(\mathcal H)$ is a closed two-sided ideal in $\mathcal B(\mathcal H)$. For $a\in \mathcal B(\mathcal H)$, we let $[a]$ denote the quotient of $a$ in the Calkin algebra ${\mathcal C}(\mathcal H):=\mathcal B(\mathcal H)/{\mathcal K}(\mathcal H)$. Thus, the essential projective spectrum $P_e(A)$ is simply the projective spectrum of $[A_1], ..., [A_n]$ in ${\mathcal C}(\mathcal H)$.

 Lemma \ref{spec} asserts that, for $\varphi \in C(\overline{ \mathbb {D}})$, the Toeplitz operator $T_\varphi$ is Fredholm if and only if $\varphi$ does not vanish on $\mathbb T$. Furthermore, in this case its index
\begin{eqnarray*}
\text{ind}(T_\varphi )= \mathrm{dim~ker} (T_\varphi)- \mathrm{dim~ker} (T^*_\varphi)= - \frac{1}{2\pi \text{i}}\int_{\varphi (\mathbb T)} \frac{1}{w} dw.
\end{eqnarray*}
When $\varphi$ is meromorphic, we have
 \[ \mathrm{ind}~(T_\varphi ) =- \frac{1}{2\pi \text{i}}\int_{\mathbb T} d\log \varphi,\] which is equal to the total number of poles minus the total number of zeros inside $\mathbb {D}$. In the case $\varphi (z)=\gamma \overline{z}^m + \alpha z^m + \beta$, and $z\in \mathbb T$, we can write

 \begin{eqnarray}\label{ind}
\varphi (z)= \frac{\alpha z^{2m}+ \beta z^m +\gamma}{z^m}=: \frac{\phi(z^m)}{z^m},
\end{eqnarray}
which extends meromorphically to $\mathbb{C}$. Obviously, the rational function $\frac{\phi(z^m)}{z^m}$ has a pole at $0$ of multiplicity $m$, and two zeros in $\mathbb {C}$, each having multiplicity $m$. Thus, assuming $\phi$ does not vanish on $\mathbb T$, there are three possibilities: $\phi$ has no zeros, $m$ zeros, or $2m$ zeros in $\D$. Correspondingly, $\ind(T_\varphi )$ takes values $m$, $0$ or $-m$. For simplicity, we set $t=z^m$, and for $i=0, 1, 2$, define
\[\Omega_{i}=\big\{(\alpha, \beta, \gamma)\in {\C}^3:\ \alpha t^2+\beta t+\gamma\ \text{has}~ i ~\text{zeros ~in}~ \mathbb{D}  ~~~\text{and~has ~no~ zeros~ on }~\T \}. \]
We first check $\Omega_{i}, i = 0, 1, 2$ are path-connected. More generally, for a general $n$-degree polynomial $P_z(t)=t^n+ z_1t^{n-1}+ z_2t^{n-2}+\cdots+z_n$ and $0\leq k \leq n$, we define 
\begin{align}\label{connect}
 \Omega^n_k: = \big\{ z\in \C^n: P_z(t)~\text{has}~ k ~\text{zeros~ in } ~\D ~~~\text{and~has ~no~ zeros~ on }~\T \big\}. 
\end{align}
Then we have the following fact.
\begin{lemma}\label{p connect}
$ \Omega^n_k $ is path-connected for each $n$ and $k$.
\end{lemma}
\begin{proof}
For each $z \in \Omega^n_k$ with $ k \ge 1$, let $\lb$ be a zero of $P_z(t)$ such that $|\lb| < 1$.
Thus we can write \[ P_z(t)= ( t- \lb ) ( t^{n-1}+ \alpha_1t^{n-2} + \alpha_2t^{n-3} + \cdots + \alpha_{n-1}),\]
where $z_1= \alpha_1 - \lb$, $z_j = \alpha_j - \lb \alpha_{j-1}$, for $2 \le j \le n-1$, and $z_n = -\lb \alpha_{n-1}.$
Therefore,
\begin{eqnarray*}
\begin{pmatrix}
-\lb & 1 & 0 &\cdots  \\
0 & -\lb & 1& \cdots  \\
\vdots & \vdots & \ddots & \vdots \\
0 & 0 & \cdots & -\lb
\end{pmatrix}
\begin{pmatrix}
1  \\
\alpha_1\\
\vdots\\
\alpha_{n-1}
\end{pmatrix}
:=J(-\lb) \begin{pmatrix}
1  \\
\alpha_1\\
\vdots\\
\alpha_{n-1}
\end{pmatrix}
= \begin{pmatrix}
z_1  \\
z_2\\
\vdots\\
z_n
\end{pmatrix}.
\end{eqnarray*}
Clearly, $(\alpha_1, \dots, \alpha_{n-1}) \in \Omega^{n-1}_{k-1}$. Therefore $ \Omega^n_k = \bigcup_{\lb \in \D }J(-\lb) \Omega^{n-1}_{k-1}.$
Hence if $\Omega^{n-1}_{k-1}$ is path-connected, then so is $\Omega^n_k~(k \ge 1).$
Thus, the problem is reduced whether $\Omega^n_0 $ is path-connected for each $n \ge 0.$
Using ideas similar to the above argument, we can check that $\Omega^n_0 = \bigcup_{|\lb | >1}J(-\lb) \Omega^{n-1}_0.  $
Hence if $ \Omega^{n-1}_0 $ is path-connected, then $\Omega^n_0 $ is path-connected.
So the problem is  reduced to whether $\Omega^1_0 $ is path-connected, which is clearly true.
\end{proof}

To be consistent in notations, we write $T_\varphi$ as the linear pencil $A (\Lambda):= \gamma  T_{\overline{z}^m} + \alpha  T_{z^m} + \beta  I$ with $\Lambda=( \alpha,~\beta,~\gamma) \in {\C}^3$. Then, combining \eqref{num} and the above lemma, we have the following theorem.% under the condition of $|\alpha|\neq 1$~and~$1-|\alpha|^2 \neq|\alpha \overline{\beta} - \beta|$.
\begin{theorem}%----------------------------- index
Assume $A (\Lambda)= \gamma  T_{\overline{z}^m} + \alpha  T_{z^m} + \beta  I$ and $P^c_e(A):=\C^3\setminus P_e(A) = \Omega_0 \cup \Omega_1  \cup  \Omega_2.$ Then the following statements hold.

$\mathrm{(1)}$~$$ \Omega_0= \big\{ ( \alpha,~\beta,~\gamma) \in {\C}^3:   |\gamma|^2-|\alpha|^2>|\alpha \overline{\beta} -\beta\overline{\gamma}|~\big\},$$
and on $ \Omega_0$, we have
 $$\mathrm{ind} (T_{\gamma \overline{z}^m + \alpha z^m + \beta}) = m.$$

$\mathrm{(2)}$~
$$ \Omega_1= \big\{ ( \alpha,~\beta,~\gamma) \in {\C}^3: 2|\alpha|^2+2|\gamma|^2<|\beta|^2+|\beta^2-4\alpha \gamma| \big\},$$
and on $ \Omega_1$, we have
$$\mathrm{ind} (T_{\gamma \overline{z}^m + \alpha z^m + \beta}) = 0.$$
Moreover, $ T_{\gamma \overline{z}^m + \alpha z^m + \beta}$ is invertible on $ \Omega_1$.

$\mathrm{(3)}$~ 
$$\Omega_2= \big\{ ( \alpha,~\beta,~\gamma) \in {\C}^3: |\alpha|^2 - |\gamma|^2> |\alpha \overline{\beta} - \beta\overline{\gamma}|~\big\},$$
and on $ \Omega_2$, we have
 $$\mathrm{ind} (T_{\gamma \overline{z}^m + \alpha z^m + \beta}) = -m.$$
\end{theorem}
%------------------------------------------proof of (1)
\begin{proof}

$\mathrm{(1)}$~Since $( \alpha,~\beta,~\gamma) \in \Omega_0$, we have that the zeros $t_0$ and $t_1$ of $\alpha t^2+\beta t+\gamma$ are outside $\D$.
By Vieta's formulas, $\frac{1}{t_0}$ and $\frac{1}{t_1}$ are the roots of  $ \gamma t^2+\beta t + \alpha =0$.
The application of Theorem \ref{dis2}  gives $|\gamma|^2-|\alpha|^2>|\alpha \overline{\beta} -\beta\overline{\gamma}|.$

For the remaining part of statement $\mathrm{(1)}$, by Lemma \ref{p connect} and the continuity of index, it suffice to check the index for any fixed point in $\Omega_0$. When $\beta =0$ and $|\gamma| > |\alpha|$, using Theorem \ref{cor3}, $\mathrm{dim~ker} (T_\varphi)= m$ and $\mathrm{dim~ker}(T^*_\varphi)=0$.
It follows that $\mathrm{ind} A(\Lambda) = m$ on $\Omega_0$. We complete the proof.

%

%------------------proof of (2)

$\mathrm{(2)}$~~Let $( \alpha,~\beta,~\gamma) \in \Omega_1$. Then $\alpha t^2+\beta t+\gamma$ has exactly one zero inside $\D$ and has no zeros on $\T$.

When $\alpha =0$, we have the modulus of the zero $|t|=|\frac{-\gamma}{\beta}| < 1$, this is $|\gamma| < |\beta|$.

When $\alpha \neq 0$, then $\alpha t^2+\beta t+\gamma$ has one zero inside $\D$ and one zero outside $\D$. By Vieta's theorem, we can get two roots with  $t_0 =\frac{-\beta + \sqrt{\beta^2- 4\alpha \gamma}}{2 \alpha }$ and $t_1 =\frac{-\beta - \sqrt{\beta^2- 4\alpha \gamma}}{2 \alpha}$,  here $\sqrt{\beta^2- 4\alpha \gamma}$ means the usual principal branch. Hence either $|t_0|<1<|t_1|$ or $|t_1|<1<|t_0|$. It implies that 
$$|\beta|^2+|\beta^2-4\alpha \gamma|-2\mathrm{Re}\left(\overline{\beta}\sqrt{\beta^2- 4\alpha \gamma}\right)<4|\alpha|^2<|\beta|^2+|\beta^2-4\alpha \gamma|+2\mathrm{Re}\left(\overline{\beta}\sqrt{\beta^2- 4\alpha \gamma}\right)$$
or 
$$|\beta|^2+|\beta^2-4\alpha \gamma|+2\mathrm{Re}\left(\overline{\beta}\sqrt{\beta^2- 4\alpha \gamma}\right)<4|\alpha|^2<|\beta|^2+|\beta^2-4\alpha \gamma|-2\mathrm{Re}\left(\overline{\beta}\sqrt{\beta^2- 4\alpha \gamma}\right).$$
This is equivalent to 
$$ \bigg| 4|\alpha|^2-|\beta|^2-|\beta^2-4\alpha \gamma| \bigg|< \bigg| 2\mathrm{Re}\left(\overline{\beta}\sqrt{\beta^2- 4\alpha \gamma}\right) \bigg| .$$
Since for a complex number $z$, $|2\mathrm{Re} z|^2=(z+\bar{z})^2=2\mathrm{Re}z^2 +2|z|^2$, we can rewrite the above inequality as 
$$ \left( 4|\alpha|^2-|\beta|^2-|\beta^2-4\alpha \gamma| \right)^2 <2|\beta|^2 |\beta^2-4\alpha \gamma| +2\mathrm{Re} \left( |\beta|^4-4\overline{\beta}^2\alpha \gamma \right).$$
This is equivalent to
$$ 2|\alpha|^2+2|\gamma|^2<|\beta|^2+|\beta^2-4\alpha \gamma|. $$
Note that when $\alpha=0$, the above inequality is reduced to $|\gamma|<|\beta|$, hence in all, we conclude that 
$$ \Omega_1= \big\{ ( \alpha,~\beta,~\gamma) \in {\C}^3: 2|\alpha|^2+2|\gamma|^2<|\beta|^2+|\beta^2-4\alpha \gamma| \big\}.$$

For the remaining part of statement $\mathrm{(2)}$, again by Lemma \ref{p connect} and the continuity of index, it suffice to check the index for any fixed point in $\Omega_1$. If $\alpha =0$ and $|\gamma| < |\beta|$, then we have $\mathrm{ker}\left( T_\varphi\right)= \{0\}$ by Lemma \ref{poly ker}.
Since $T^*_\varphi$ is an analytic Toeplitz operator, we have $\ker \left(T^*_\varphi\right)= \{0\}.$
Hence  $\mathrm{ind} A(\Lambda) = 0$ on $\Omega_1$. Moreover, by the proof of Theorem \ref{main m}, we conclude that $A(\Lambda)$ is always injective, therefore $A(\Lambda)$ is invertible on $\Omega_1$.

%------------------------- proof of (3)
$\mathrm{(3)}$~If $( \alpha,~\beta,~\gamma) \in \Omega_2$, we have that the zeros  of $\alpha t^2+\beta t+\gamma$ are all inside $\D$.
By  Theorem \ref{dis2}, we have  $|\alpha|^2 - |\gamma|^2> |\alpha \overline{\beta} - \beta\overline{\gamma}|.$

For the remaining part of statement $\mathrm{(3)}$, again by Lemma \ref{p connect} and the continuity of index, it suffice to check the index for any fixed point in $\Omega_2$. When $\beta =0$ and $|\alpha| > |\gamma |$,  Theorem \ref{cor3} gives $ \ker \left(T_\varphi\right)= 0$ and $\dim \ker \left(T^*_\varphi\right)=m$.
It follows that $\mathrm{ind} A(\Lambda) = -m$ on $\Omega_2$. The proof is completed.

\end{proof}%------------------ finish the proof
%\begin{remark}
%Suppose $\gamma =1$,  $|\alpha|= 1$~and~$ 1-|\alpha|^2 =|\alpha \overline{\beta} - \beta| $.
%Set $\alpha = e^{\mathrm{i}2\theta}$. Then simple calculations imply $ \beta = |\beta|e^{\mathrm{i}\theta}$.
%At this moment, $\varphi (z)= \overline{z}^m +  e^{\mathrm{i}2\theta} z^m + |\beta|e^{\mathrm{i}\theta}= e^{\mathrm{i}\theta}~\big( e^{-\mathrm{i}\theta} \overline{z}^m +  e^{\mathrm{i}\theta} z^m + |\beta| \big):= e^{\mathrm{i}\theta}\varphi_1 (z) $~and~$T_{\varphi}=e^{\mathrm{i}\theta} T_{\varphi_1}$.
%$T_{\varphi_1}$ is an adjoint Toeplitz operator with  bounded real harmonic symbol and has  nonzero eigenvectors, see \cite{MC1979}.
%The essential spectrum of $ T_{\varphi_1} $ is $[ \inf \varphi_1, \sup \varphi_1 ]$, refer to \cite{MC1979}.
%\end{remark}
%-----------cor kernel

Based on the proof of the above theorem and Theorem \ref{injective}, for Toeplitz operator $T_{\varphi}$ with $\varphi (z)= \gamma \overline{z}^m + \alpha z^m + \beta$,
 it is only when the zeros $t_0$ and $t_1$ of $\alpha t^2+\beta t+\gamma$ both lie on $\T$~that~$T_{\varphi}$~may not be injective.
In this case, $|t_0t_1|= |\frac{\gamma}{\alpha}|=1$.
Set \[ \alpha = e^{\mathrm{i}2\theta}\gamma =  e^{\mathrm{i}(2\theta + \theta_1)}|\gamma|~ \text{and}~\beta = |\beta|e^{\mathrm{i}\theta_2}.\]
For the case $\beta =0$, see Theorem \ref{cor3}.
We consider the case $\beta \neq 0$.
Since $|t_0|^2 = |t_1|^2 =1$, we have  $ \mathrm{Re}~\big(\overline{\beta}\sqrt{\beta^2- 4\alpha \gamma} \big) =0$ by the proof of above theorem.
Set $\overline{\beta}\sqrt{\beta^2- 4\alpha \gamma} = iy$ with $y \in \R$.
It follows that $ |\beta|^2 - 4 {\overline{\beta}}^2\alpha \gamma  =-y^2$.
Therefore, $$ 4{\overline{\beta}}^2\alpha \gamma = |\beta|^2 + y^2 =4|\beta\gamma^2|e^{\mathrm{i}(2\theta +2 \theta_1-2 \theta_2)} \ge 0.$$
We conclude that $ \frac{\beta}{\gamma}e^{-\mathrm{i\theta}}= |\frac{\beta}{\gamma}|e^{\mathrm{-i}(\theta + \theta_1- \theta_2)} \in \R.$
In this case
$$ \varphi (z) = \gamma \overline{z}^m + \alpha z^m + \beta =\gamma e^{\mathrm{i}\theta} \big( e^{-\mathrm{i}\theta} \overline{z}^m + e^{\mathrm{i}\theta}  z^m +  \frac{\beta}{\gamma}e^{-\mathrm{i}\theta} \big):=\gamma e^{\mathrm{i}\theta} \varphi_1 (z). $$
Then
$ T_{\varphi} =  \gamma e^{\mathrm{i}\theta}T_{ \varphi_1 }$,  where $T_{\varphi_1}$ is a Toeplitz operator with  bounded real harmonic symbol.
Since $\gamma \neq 0$, we have $\mathrm{ker}\left( T_\varphi\right)=\mathrm{ker} \left(T_{\varphi_1}\right)$.
If there exist $g(z) \in  L^2_a({\mathbb D})$ such that $T_{\varphi_1}g =0$, then $$ T_{e^{-\mathrm{i}\theta} \overline{z}^m + e^{\mathrm{i}\theta}  z^m}g = -\frac{\beta}{\gamma}e^{-\mathrm{i}\theta} g~(\beta \neq 0).$$
Since  a Toeplitz operator with the bounded real harmonic symbol  has  nonzero eigenvectors \cite[Proposition 13]{MC1979},
we have $\mathrm{ker}\left( T_\varphi\right)=\{0\}$.
Thus, we arrive at the following  corollary.
\begin{corollary}
Suppose that $\varphi (z)= \gamma \overline{z}^m + \alpha z^m + \beta$.
 Then $T_{\varphi}$ is always of Coburn type.
\end{corollary}

%--------------references--------------------------------------------

%\bibliographystyle{amsplain}

\end{document}